\documentclass[final]{siamart190516}
\usepackage{graphicx,epstopdf,epsfig,amsmath,amssymb,amsfonts,bm}
\usepackage{cleveref,hyperref,colortbl}
\usepackage{algorithm}
\usepackage[noend]{algorithmic}
\usepackage{soul}
\usepackage[ntheorem]{empheq}
\soulregister\cite7
\soulregister\cref7
\soulregister\ref7

\definecolor{myca}{rgb}{0.9686,0.5765,0.1176}
\definecolor{mycb}{rgb}{0.4000,0.7059,0.9098}
\definecolor{mycc}{rgb}{0.9765,0.7686,0.6000}
\definecolor{mycd}{rgb}{0.8510,0.5843,0.5608}

\begin{document}
\title{Searching the solution landscape by generalized high-index saddle dynamics}

\author{
Jianyuan Yin\thanks{School of Mathematical Sciences, Peking University, Beijing 100871, China (\email{yinjy@pku.edu.cn}).}
\and
Bing Yu\thanks{School of Mathematical Sciences, Peking University, Beijing 100871, China (\email{yubing93@pku.edu.cn}).}
\and
Lei Zhang\thanks{Beijing International Center for Mathematical Research, Center for Quantitative Biology, Peking University, Beijing 100871, China (\email{zhangl@math.pku.edu.cn}).}
}

\maketitle

\begin{abstract}
  We introduce a generalized numerical algorithm to construct the \emph{solution landscape}, which is a pathway map consisting of all stationary points and their connections.
  Based on the high-index optimization-based shrinking dimer (HiOSD) method for gradient systems, a generalized high-index saddle dynamics (GHiSD) is proposed to compute any-index saddles of dynamical systems.
  Linear stability of the index-$k$ saddle point can be proved for the GHiSD system.
  A combination of the downward search algorithm and the upward search algorithm is applied to systematically construct the solution landscape, which not only provides a powerful and efficient way to compute multiple solutions without tuning initial guesses, but also reveals the relationships between different solutions. 
  Numerical examples, including a three-dimensional example and the phase field model, demonstrate the novel concept of the solution landscape by showing the connected pathway maps.
\end{abstract}



\section{\label{sec:intro}Introduction}

The \emph{energy landscape}, which is a mapping of all possible configurations of the system to their energy, exhibits a number of local minima separated by barriers.
The energy landscape has been widely devoted to elucidating the structure and thermodynamics of the energy functions in a broad range of applications, such as protein folding \cite{leeson2000protein, mallamace2016energy, onuchic1997theory, wales2003energy}, catalysis \cite{benkovic2008free, kerns2015energy}, Lennard--Jones clusters \cite{cai2018single, wales1997global}, phase transitions \cite{cheng2010nucleation, du2020variational, han2019transition}, and artificial neural networks \cite{das2016energy, draxler2018essentially, weinan2020a, goodfellow2016deep}.
The index-$1$ saddle point, often characterized as the transition state, is located at the point of the minimal energy barrier between two minima.
The minimum energy path connects two minima and the transition state via a continuous curve on the energy landscape \cite{weinan2002string, Ren2013a}.
Besides minima and transition states, the stationary points of the energy function also include high-index saddle points.
To characterize the nature of a nondegenerate saddle point, the Morse index of a saddle point is the maximal dimension of a subspace on which its Hessian is negative definite \cite{milnor1963morse}.
An intriguing mathematical-physics problem is to efficiently search all stationary points of a multivariable energy function, including both minima and saddle points \cite{ hughes2014an, mehta2011finding}.
A large number of numerical methods have been proposed to compute the minima and index-$1$ saddle points on complicated energy landscapes in the recent two decades \cite{du2009constrained, vanden2010transition, henkelman2002methods, yu2020global, zhang2016optimization, zhang2016recent}.

Different from the energy landscape, we introduce a novel and more informative concept of the \emph{solution landscape}, which is a pathway map consisting of all stationary points and their connections.
The solution landscape can present how minima are connected with index-1 saddle points, and how lower-index saddle points are connected to higher-index ones, finally to a parent state, an index-$k$ saddle point ($k$-saddle), as shown in \cref{fig1}(A).
We illustrate the solution landscape with a two-dimensional example,
\begin{equation}\label{2dE}
E(x, y) = (x^2-1)^2 + (y^2-1)^2.
\end{equation}
According to Morse indices, the solutions to $\nabla E(x, y)= \bm 0$ are classified as: four minima $(\pm 1, \pm 1)$, four 1-saddles $(\pm 1, 0) $ and $(0, \pm 1)$, and one maximum $(0,0)$.
From the maximum (2-saddle), the solution landscape can be constructed down to four minima by following unstable directions.
Specifically, four 1-saddles can found from the 2-saddle along its two unstable eigenvectors of the Hessian, and each 1-saddle connects two minima by following the gradient flows from both sides of the unstable eigenvector, as shown in \cref{fig1}(B).

\begin{figure}[h]
	\centering
	\includegraphics[width=.6\columnwidth]{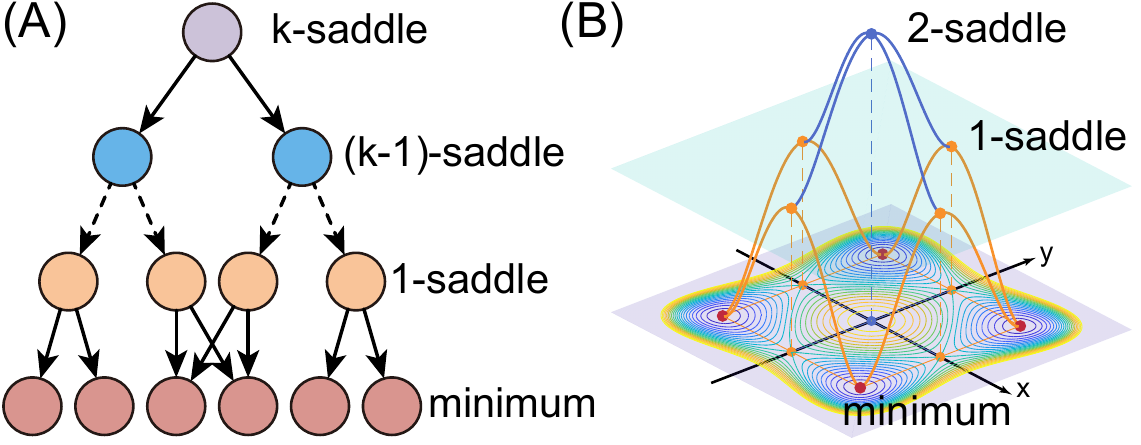}\\
	\caption{(A) A diagram of the solution landscape.
		(B) The solution landscape of the function \cref{2dE}. }\label{fig1}
\end{figure}

The above definition of the solution landscape is intuitive, while it is challenging to define and construct the solution landscape numerically.
Multiple stationary points can be found by numerically solving the nonlinear equations $\nabla E(\bm x)=\bm 0$ with extensive algorithms, such as homotopy methods \cite{chen2004search, hao2014bootstrapping, mehta2011finding} and deflation techniques \cite{brow1971deflation, farrell2015deflation}.
However, as more solutions are discovered, it becomes increasingly difficult to tune fine initial guesses for the remaining solutions and determine whether all solutions have been found.
Moreover, the relationships between all solutions are not portrayed with these methods.
Recently, Yin et al. proposed a numerical procedure to construct a pathway map on the energy landscape based on the high-index optimization-based shrinking dimer (HiOSD) method \cite{yin2020construction, yin2019high}.
A downward search enables to construct the pathway map from a $k$-saddle to the connected lower-index saddles, and the HiOSD method also embeds an upward search with a selected direction to find a higher-index saddle, so the entire search can navigate up and down on the energy landscape.

Generally, a dynamical system, $\dot{\bm x}=\bm F(\bm x)$, is an evolution rule that defines a trajectory as a function of time on a set of states (the phase space) \cite{meiss2007differential}.
It has been widely applied in depicting the motions of physics, chemistry and biology, such as kinetic equations \cite{chen2003extended, shakhov1968generalization}, Navier-Stokes equations \cite{lemarie2002recent, temam2001navier} and biochemical reactions \cite{chen2010classic, nie2020noise, qiao2019network}.
Similar to gradient systems, an index-$k$ saddle point is a stationary point where the Jacobian has exactly $k$ eigenvalues with a positive real part \cite{wiggins2003introduction}.
Although the energy landscape no longer exists in general dynamical systems, the solution landscape is generic and valid in both gradient and non-gradient systems.
However, the original procedure presented cannot deal with non-gradient systems and needs to be generalized.

In this paper, we introduce a generalized numerical method to construct the solution landscape for a dynamical system.
A generalized high-index saddle dynamics (GHiSD) method is proposed to find high-index saddles of non-gradient systems and linear stability of the index-$k$ saddle point is proved for the GHiSD system.
Based on the GHiSD method, we develop a systematic approach by a combination of the downward search and the upward search to efficiently construct the solution landscape, starting from a high-index saddle and ending with multiple sinks.
We apply a three-dimensional example and the phase field model as numerical examples to demonstrate the novel concept of the solution landscape.

\section{GHiSD Methods}\label{sec:method}
\subsection{Dynamical systems}
Given an autonomous dynamical system \cite{wiggins2003introduction}
\begin{equation}\label{dynamical}
\dot{\bm x}=\bm F(\bm x), \quad \bm x\in \mathbb{R}^n,
\end{equation}
where $\bm F: \mathbb{R}^n\rightarrow\mathbb{R}^n$ is a $\mathcal{C}^r(r \geqslant 2)$ function, a point ${\bm x}^*\in\mathbb{R}^n$ is called a \emph{stationary point} (or \emph{equilibrium solution}) of \cref{dynamical} if $\bm F(\bm x^*)=\bm 0$.
Let $J(\bm x)=\nabla \bm F(\bm x)$ denote the Jacobian of $\bm F(\bm x)$, and $\langle \cdot,\cdot\rangle$ denote an inner product on $\mathbb{R}^n$, i.e. $\langle \bm x,\bm y\rangle= \bm y^\top \bm x$.

For a stationary point $\bm x^\ast$, taking $\bm x=\bm x^*+\bm y$ in \cref{dynamical}, we have
\begin{equation}\label{doty}
\dot{\bm y}=J(\bm x^*)\bm y+\mathcal{O}(\|\bm y\|^2),
\end{equation}
where $\|\cdot\|$ denotes the norm induced by the inner product.
The associated linear system
\begin{equation}\label{linear}
\dot{\bm y}=J(\bm x^*)\bm y
\end{equation}
is used to determine the stability of $\bm x^*$.
Let $\{\bm w_1, \cdots, \bm w_{k_{\mathrm{u}}}\}$, $\{\bm w_{k_{\mathrm{u}}+1}, \cdots, \bm w_{k_{\mathrm{u}}+k_{\mathrm{s}}}\}$ and $\{\bm w_{k_{\mathrm{u}}+k_{\mathrm{s}}+1}, \cdots, \bm w_{k_{\mathrm{u}}+k_{\mathrm{s}}+k_{\mathrm{c}}}\}\subset\mathbb{C}^n$ denote the (generalized) right eigenvectors of $J(\bm x)$ corresponding to the eigenvalues of $J(\bm x)$ with positive, negative and zero real parts, respectively, where $k_{\mathrm{u}}+k_{\mathrm{s}}+k_{\mathrm{c}}=n$.
With these eigenvectors, we define unstable, stable and center subspaces of the Jacobian $J(\bm x)$, respectively, as $\mathcal{W}^{\mathrm{u}}(\bm x)$, $\mathcal{W}^{\mathrm{s}}(\bm x)$ and $\mathcal{W}^{\mathrm{c}}(\bm x)$:
\begin{equation}\label{suc}
\begin{aligned}
\mathcal{W}^{\mathrm{u}}(x) &= \mathrm{span}_{\mathbb{C}}
\{\bm w_1,\cdots,\bm w_{k_{\mathrm{u}}}\}\cap\mathbb{R}^n, \\
\mathcal{W}^{\mathrm{s}}(x) &= \mathrm{span}_{\mathbb{C}}
\{\bm w_{k_{\mathrm{u}}+1},\cdots,\bm w_{k_{\mathrm{u}}+k_{\mathrm{s}}}\}\cap\mathbb{R}^n, \\
\mathcal{W}^{\mathrm{c}}(x) &= \mathrm{span}_{\mathbb{C}}
\{\bm w_{k_{\mathrm{u}}+k_{\mathrm{s}}+1},\cdots,\bm w_{k_{\mathrm{u}}+k_{\mathrm{s}}+k_{\mathrm{c}}}\}\cap\mathbb{R}^n.\\
\end{aligned}
\end{equation}
From the primary decomposition theorem \cite{hirsch1974differential}, $\mathbb{R}^n$ can be decomposed as a direct sum:
\begin{equation}\label{RnM}
\mathbb{R}^n=
\mathcal{W}^{\mathrm{u}}(\bm x) \oplus
\mathcal{W}^{\mathrm{s}}(\bm x) \oplus
\mathcal{W}^{\mathrm{c}}(\bm x).
\end{equation}
By treating $J(\bm x^\ast)$ as a linear operator on $\mathbb{R}^n$, $\mathcal{W}^{\mathrm{u}}(\bm x^*)$, $\mathcal{W}^{\mathrm{s}}(\bm x^*)$ and $\mathcal{W}^{\mathrm{c}}(\bm x^*)$ are the invariant subspaces of the linear system \cref{linear}.

The nonlinear system \cref{doty} possesses
a $k_{\mathrm{u}}$-dimensional local invariant unstable manifold $\mathcal{M}^{\mathrm{u}}_{\mathrm{loc}}(\bm x^*)$,
a $k_{\mathrm{s}}$-dimensional local invariant stable manifold $\mathcal{M}^{\mathrm{s}}_{\mathrm{loc}}(\bm x^*)$, and
a $k_{\mathrm{c}}$-dimensional local invariant center manifold $\mathcal{M}^{\mathrm{c}}_{\mathrm{loc}}(\bm x^*)$ near the origin.
They are intersecting and tangent to the respective invariant subspaces $\mathcal{W}^{\mathrm{u}}(\bm x^*)$, $\mathcal{W}^{\mathrm{s}}(\bm x^*)$, $\mathcal{W}^{\mathrm{c}}(\bm x^*)$ \cite{wiggins2003introduction}.
``Local'' here refers to that these manifolds with boundaries are defined only in a neighborhood of $\bm x^*$.
The invariant subspaces of the linear system \cref{linear} can be used to study the stationary point of the nonlinear system \cref{doty}.

A stationary point $\bm x^*$ of \cref{dynamical} is called \emph{hyperbolic} if no eigenvalues of $J(x^*)$ have zero real part, which means $\mathcal{W}^{\mathrm{c}}(\bm x^*)=\{\bm 0\}$ and the decomposition \cref{RnM} can be rewritten as
\begin{equation}\label{RnE2}
\mathbb{R}^n=\mathcal{W}^{\mathrm{u}}(\bm x^*) \oplus \mathcal{W}^{\mathrm{s}}(\bm x^*).
\end{equation}
A hyperbolic stationary point is called a \emph{saddle} if $\mathcal{W}^{\mathrm{u}}(\bm x^*)$ and $\mathcal{W}^{\mathrm{s}}(\bm x^*)$ are nontrivial.
The hyperbolic stationary point $x^*$ is called a \emph{sink} (\emph{source}) if all the eigenvalues of $J(\bm x^*)$ have negative (positive) real parts.
The \emph{index} of a stationary point $\bm x^\ast$ is defined as the dimension of the unstable subspace $\mathcal{W}^{\mathrm{u}}(\bm x^*)$ \cite{wiggins2003introduction}.

\subsection{Review of HiOSD method for gradient systems}
The HiOSD method is designed for finding index-$k$ saddles of an energy function $E(\bm x)$ \cite{yin2019high}.
For gradient systems $\bm F(\bm x)=-\nabla E(\bm x)$, the Jacobian $J(\bm x)=-\nabla^2 E(\bm x)=-G(\bm x)$ is self-adjoint with all eigenvalues real, where $G(\bm x)$ denotes the Hessian of $E(\bm x)$.
The hyperbolic $k$-saddle $\bm x^*$ is a local maximum on the linear manifold $\bm x^*+\mathcal{W}^{\mathrm{u}}(\bm x^*)$ and a local minimum on $\bm x^*+ \mathcal{W}^{\mathrm{s}}(\bm x^*)$.

Regardless of the dimer approximation temporarily, the high-index saddle dynamics (HiSD) for a $k$-saddle ($k$-HiSD) has the following form:
\begin{subequations}\label{HiOSD1}
	\begin{empheq}[left=\empheqlbrace]{align}
	\dot{\bm x}   & = \left(I-2\sum_{j=1}^{k}\bm v_j \bm v_j^\top\right) \bm F(\bm x), \label{HiOSD1x}\\
	\dot{\bm v}_i & = -\left(I-\bm v_i\bm v_i^\top-2\sum_{j=1}^{i-1}\bm v_j \bm v_j^\top\right)G(\bm x)\bm v_i,\; i=1,\cdots,k,\label{HiOSD1v}
	\end{empheq}
\end{subequations}
which involves a position variable $\bm x$ and $k$ direction variables $\bm v_i$.
The dynamics \cref{HiOSD1} is coupled with an initial condition:
\begin{equation}\label{inits}
\bm x(0)= \bm x^{(0)}\in\mathbb{R}^n, \quad \bm v_i(0)=\bm v_i^{(0)}\in\mathbb{R}^n, \qquad \mathrm{s.t.} \quad \left\langle \bm v_j^{(0)}, \bm v_i^{(0)}\right\rangle=\delta_{ij}, \quad i,j=1,\cdots,k.
\end{equation}
The dynamics \cref{HiOSD1x} represents a transformed gradient flow:
\begin{equation}\label{minimax_dynamic2}
\dot{\bm x}
= -\mathcal{P}_{\mathcal{W}^{\mathrm{u}}(\bm x)}\bm F(\bm x)+
\left(\bm F(\bm x)-\mathcal{P}_{\mathcal{W}^{\mathrm{u}}(\bm x)}\bm F(\bm x)\right)
= \left(I-2\mathcal{P}_{\mathcal{W}^{\mathrm{u}}(\bm x)}\right)\bm F\bm (x),
\end{equation}
where $\mathcal{P}_{\mathcal{V}}$ denotes the orthogonal projection operator on a finite-dimensional subspace $\mathcal{V}$. Here, $-\mathcal{P}_{\mathcal{W}^{\mathrm{u}}(\bm x)}\bm F(\bm x)$ is taken as an ascent direction on the subspace $\mathcal{W}^{\mathrm{u}}(\bm x)$ and $\bm F(\bm x)-\mathcal{P}_{\mathcal{W}^{\mathrm{u}}(\bm x)}\bm F(\bm x)$ is a descent direction on the subspace $\mathcal{W}^{\mathrm{s}}(\bm x)$.
The dynamics \cref{HiOSD1v} finds an orthonormal basis of $\mathcal{W}^{\mathrm{u}}(\bm x)$.
Thanks to the self-adjoint Jacobian $J(\bm x)$, we can simply take $\bm v_i$ as a unit eigenvector corresponding to the $i$th smallest eigenvalue of $G(\bm x)$.
Therefore, the $i$th eigenvector $\bm v_i$ can be obtained by a constrained optimization problem with the knowledge of $\bm v_1,\cdots,\bm v_{i-1}$:
\begin{equation}\label{minvi}
\min_{\bm v_i\in\mathbb{R}^n} \quad
\langle G(\bm x)\bm v_i, \bm v_i\rangle,\qquad
\mathrm{s.t.} \quad
\langle \bm v_j,\bm v_i\rangle=\delta_{ij},\quad j=1,2,\cdots,i.
\end{equation}
Then the $k$ Rayleigh quotients \cref{minvi} are minimized simultaneously using the gradient flow of $\bm v_i$:
\begin{equation}\label{dotv2}
\dot{\bm v}_i = -G(\bm x)\bm v_i +
\langle G(\bm x)\bm v_i, \bm v_i\rangle \bm v_i
+2\sum_{j=1}^{i-1}\left\langle G(\bm x)\bm v_i, \bm v_j\right\rangle \bm v_j, \quad i=1,\cdots,k,
\end{equation}
as the dynamics \cref{HiOSD1v}.

\subsection{GHiSD for non-gradient systems}
The GHiSD for a $k$-saddle ($k$-GHiSD) of the dynamical system \cref{dynamical} has the following form:
\begin{subequations}\label{HiOSD2}
	\begin{empheq}[left=\empheqlbrace]{align}
	\dot{\bm x}   & = \left(I-2\sum_{j=1}^{k}\bm v_j \bm v_j^\top\right) \bm F(\bm x), \label{HiOSD2x}\\
	\dot{\bm v}_i & = \left(I-\bm v_i \bm v_i^\top\right)J(\bm x)\bm v_i
	-\sum_{j=1}^{i-1} \bm v_j \bm v_j^\top \Big(J(\bm x)+J^\top(\bm x)\Big)\bm v_i,
	\quad i=1,\cdots,k,\label{HiOSD2v}
	\end{empheq}
\end{subequations}
with an initial condition \cref{inits}.

The dynamics \cref{HiOSD2x} is the same as the dynamics \cref{HiOSD1x} of HiSD.
For non-gradient systems, the Jacobian $J(\bm x)$ is not self-adjoint, and the eigenvectors may not be orthogonal.
The loss of orthogonality between $\mathcal{W}^{\mathrm{u}}(\bm x)$ and $\mathcal{W}^{\mathrm{s}}(\bm x)$ indicates that it may be hard to decompose $\bm F(\bm x)$ according to the direct sum \cref{RnE2}.
For this reason, the gentlest ascent dynamics (GAD) for searching $1$-saddles in non-gradient systems requires two direction variables to approximate the left and right eigenvectors corresponding to the largest eigenvalue of $J(\bm x)$ \cite{weinan2011gentlest}.
Nevertheless, Gu and Zhou later proposed a simplified GAD in non-gradient systems that uses only one direction variable and proved that its stable critical point corresponds to a $1$-saddle of the original system \cite{gu2018simplified}.
Thus, we adopt the similar idea that the dynamics \cref{minimax_dynamic2} could be used to find high-index saddles of non-gradient systems and the HiSD \cref{HiOSD1x} is preserved in GHiSD \cref{HiOSD2x}, as long as $\bm v_1,\cdots,\bm v_k$ approximate an orthonormal basis of the subspace $\mathcal{W}^{\mathrm{u}}(\bm x)$.

The dynamics \cref{HiOSD2v} aims to find an orthonormal basis $\bm v_1,\cdots,\bm v_k$ of the subspace $\mathcal{W}^{\mathrm{u}}(\bm x)$.
The dynamical flow
\begin{equation}\label{dotvi2}
\dot{\bm v}_i = J(\bm x)\bm v_i + \sum_{j=1}^{i}\xi_{ij} \bm v_j,\quad i=1,\cdots,k,
\end{equation}
should preserve the orthonormal condition
\begin{equation}\label{orthonormal2}
\langle \bm v_j, \bm v_i\rangle=\delta_{ij}, \quad i,j=1,\cdots,k,
\end{equation}
which indicates
\begin{equation}\label{xi_j2}
\xi_{ii} = -\langle J(\bm x)\bm v_i, \bm v_i\rangle; \quad
\xi_{ij} = -\langle J(\bm x)\bm v_i, \bm v_j\rangle-
\langle \bm v_i, J(\bm x)\bm v_j \rangle,
\quad j=1,\cdots,i-1.
\end{equation}
Thus, the dynamics \cref{HiOSD2v} is derived from \cref{dotvi2} and \cref{xi_j2}, which also accords with \cref{HiOSD1v} under the condition $J(\bm x)=J^\top(\bm x)$.

Next, we show the following theorem for linear stability for the $k$-GHiSD \cref{HiOSD2} under some assumptions.

\begin{theorem}\label{them1}
	Assume that $\bm F(\bm x)$ is $\mathcal{C}^2$, $\bm x^*\in\mathbb{R}^n$,
	and $\{\bm v^*_i\}^k_{i=1}\subset \mathbb{R}^n$ satisfying the orthonormal condition \cref{orthonormal2}.
	The eigenvalues of the Jacobian $J^*=\nabla \bm F(\bm x^*)$ are $\lambda_1>\cdots>\lambda_k>\lambda_{k+1}\geqslant\cdots\geqslant\lambda_n$, real and nonzero,
	and the corresponding right eigenvector of $\lambda_j$ is $\bm w_j\in \mathbb{R}^n$.
	Then $(\bm x^*,\bm v_1^*,\bm v_2^*,\cdots,\bm v_k^*)$ is a linearly stable stationary point of the dynamics \cref{HiOSD2},
	if and only if $\bm x^*$ is a $k$-saddle of the dynamical system \cref{dynamical},
	and $\{\bm v^*_j\}^i_{j=1}$ is an orthonormal basis of $\mathrm{span}\{\bm w_1, \cdots, \bm w_i\}$ for $i=1,2,\cdots,k$.
\end{theorem}

\begin{proof}
	The Jacobian of the dynamics \cref{HiOSD2} is
	\begin{equation}\label{J_all}
	\mathbb{J}(\bm x, \bm v_1,\cdots, \bm v_k)=
	\dfrac{\partial(\dot{\bm x},\dot{\bm v}_1,\cdots,\dot{\bm v}_k)}{\partial(\bm x,\bm v_1,\cdots,\bm v_k)} =
	\begin{bmatrix}
	\mathbb{J}_{xx} & \mathbb{J}_{x1} & \mathbb{J}_{x2} & \mathbb{J}_{x3} & \cdots & \mathbb{J}_{xk} \\
	*               & \mathbb{J}_{11} & O               & O               & \cdots & O               \\
	*               & *               & \mathbb{J}_{22} & O               & \cdots & O               \\
	\vdots          & \vdots          & \vdots          & \vdots          &        & \vdots          \\
	*               & *               & *               & *               & \cdots & \mathbb{J}_{kk}
	\end{bmatrix},
	\end{equation}
	whose blocks have following expressions:
	\begin{equation}\label{J_parts}
	\begin{aligned}
	\mathbb{J}_{xx} & = \dfrac{\partial \dot{\bm x}}{\partial \bm x}
	= \left(I-\sum_{j=1}^{k}2\bm v_j \bm v_j^\top\right) J(\bm x),\\
	\mathbb{J}_{xi} & = \dfrac{\partial \dot{\bm x}}{\partial \bm v_i}
	= -2\left(\bm v_i \bm F(\bm x)^\top+\langle \bm F(\bm x), \bm v_i \rangle I\right),\\
	\mathbb{J}_{ii} & = \dfrac{\partial \dot{\bm v}_i}{\partial \bm v_i}
	= J(\bm x)-\left\langle J(\bm x)\bm v_i, \bm v_i\right\rangle I
	-\sum_{j=1}^{i}\bm v_j \bm v_j^\top\left(J(x)+J^\top(x)\right). \\
	\end{aligned}
	\end{equation}
	In the following, $\mathbb{J}(\bm x^*,\bm v_1^*,\bm v_2^*,\cdots,\bm v_k^*)$ is denoted as $\mathbb{J}^*$, and the blocks of $\mathbb{J}^*$ are denoted as $\mathbb{J}^*$ with corresponding subscripts.
	
	``$\Leftarrow$'':
	Because $\{\bm v^*_j\}^i_{j=1}$ is an orthonormal basis of $\mathrm{span}\{\bm w_1, \cdots, \bm w_i\}$ for $i=1,\cdots,k$, the matrix $A_k=\left[\bm v_1^*,\cdots,\bm v_k^*\right]^\top J^* \left[\bm v_1^*,\cdots,\bm v_k^* \right]$ is upper triangular.
	We first extend $\{\bm v^*_j\}^k_{j=1}$ to an orthonormal basis $\{\bm v^*_1,\cdots,\bm v^*_k,\tilde{\bm v}_{k+1},\cdots,\tilde{\bm v}_n\}$ of $\mathbb{R}^n$.
	With $\tilde{V}=\begin{bmatrix} \bm v^*_1,\cdots,\bm v^*_k,\tilde{\bm v}_{k+1},\cdots,\tilde{\bm v}_n\end{bmatrix}$, the matrix $\tilde{V}^\top J^*\tilde{V}$ can be blocked as $\begin{bmatrix} A_k & * \\ O & \tilde{B}_k \\ \end{bmatrix}$, and the eigenvalues of $\tilde{B}_k\in \mathbb{R}^{(n-k)\times(n-k)}$ are $\{\lambda_{k+1},\cdots,\lambda_{n}\}\subset \mathbb{R}$.
	From the real Schur decomposition theorem \cite{golub2012matrix}, there exists an orthogonal matrix $Q$ so that $Q^\top\tilde{B}_kQ$ is upper triangular with real diagonal entries $\{\lambda_{k+1},\cdots,\lambda_{n}\}$.
	With $[\bm v^*_{k+1},\cdots,\bm v^*_n]=[\tilde{\bm v}_{k+1},\cdots,\tilde{\bm v}_n]Q$, the matrix $A=(a_{ij})={V^*}^\top J^*{V^*}$ is an $n$-by-$n$ real upper triangular matrix with diagonal entries $a_{ii}=\lambda_i$, where $\{\bm v^*_i\}^n_{i=1}$ is an orthonormal basis of $\mathbb{R}^n$ and $V^*=\begin{bmatrix} \bm v_1^*,\cdots,\bm v_n^*\end{bmatrix}$.
	For $i=1,2,\cdots,k$, the matrix $A$ can be blocked as $A=\begin{bmatrix} A_i & C_i \\ O & B_i \\ \end{bmatrix}$, where $A_i\in\mathbb{R}^{i\times i}$ and $B_i\in\mathbb{R}^{(n-i)\times (n-i)}$ are both upper triangular with diagonal entries $\lambda_1,\cdots,\lambda_{i}$ and $\lambda_{i+1},\cdots,\lambda_{n}$ respectively.

	Since $\bm x^*$ is a $k$-saddle of the dynamical system \cref{dynamical}, we have $\bm F(\bm x^*)=\bm 0$ and $\lambda_k>0>\lambda_{k+1}$.
	Calculating \cref{HiOSD2} at $(\bm x^*,\bm v_1^*,\cdots,\bm v_k^*)$, we have $\dot{\bm x}=\bm 0$ and
	\begin{equation}\label{dotvi}
	\dot{\bm v}_i
	= \left(I-{\bm v_i^*}{\bm v_i^*}^\top \right)\sum_{j=1}^{i}a_{ji}\bm v_j^*
	-\sum_{j=1}^{i-1}(a_{ij}+a_{ji}){\bm v_j^*}=\bm 0, \quad i=1,\cdots, k,
	\end{equation}
	so $(\bm x^*,\bm v_1^*,\cdots,\bm v_k^*)$ is a stationary point of the dynamics \cref{HiOSD2}.
	Since $\bm F(\bm x^*)=\bm 0$ also indicates that $\mathbb{J}_{xi}^*$ is null and $\mathbb{J}^*$ is block lower triangular, so the eigenvalues of $\mathbb{J}^*$ are determined by $\mathbb{J}^*_{xx}$ and $\mathbb{J}^*_{ii}$.
	From
	\begin{equation}\label{VTGxxV}
	{V^*}^\top \mathbb{J}^*_{xx} V^*
	= {V^*}^\top \Big(I-2\sum_{j=1}^{k}\bm v^*_j {\bm v^*_j}^\top\Big) V^* A
	= \begin{bmatrix} -I_k & O \\ O & I_{n-k} \\ \end{bmatrix} A
	= \begin{bmatrix} -A_k & -C_k \\ O & B_k \\ \end{bmatrix},
	\end{equation}
	the eigenvalues of $\mathbb{J}^*_{xx}$ are $-\lambda_{1},\cdots,-\lambda_{k},\lambda_{k+1},\cdots,\lambda_{n}$, all negative.
	Similarly, from
	\begin{equation}\label{VTGV}
	\begin{aligned}
	{V^*}^\top \mathbb{J}^*_{ii} V^*
	&={V^*}^\top \left(J^*-\lambda_i I-\sum_{j=1}^{i}\bm v^*_j {\bm v^*_j}^\top\left(J^*+{J^*}^\top\right)\right)V^*\\
	&= A-\lambda_i I-{V^*}^\top\left(\sum_{j=1}^{i}\bm v^*_j {\bm v^*_j}^\top\right) V^* \left(A+A^\top\right) \\
	&= A-\lambda_i I-
	\begin{bmatrix} I_i & O \\ O & O \\ \end{bmatrix}
	\begin{bmatrix} A_i+ A_i^\top & C_i \\ C_i^\top  & B_i+ B_i^\top\\ \end{bmatrix}
	=\begin{bmatrix} -A_i^\top & O \\ O & B_i\\ \end{bmatrix}-\lambda_i I,
	\end{aligned}
	\end{equation}
	the eigenvalues of $\mathbb{J}^*_{ii}$ are $-\lambda_{1}-\lambda_{i}, \cdots, -\lambda_{i-1}-\lambda_{i}, -2\lambda_{i}, \lambda_{i+1}-\lambda_{i}, \cdots, \lambda_{n}-\lambda_{i}$, all negative as well.
	Since all the eigenvalues of $\mathbb{J}^*$ are negative, $(\bm x^*,\bm v_1^*,\cdots,$ $\bm v_k^*)$ is a linearly stable stationary point of the dynamics \cref{HiOSD2}.

	``$\Rightarrow$'':
	Since $(\bm x^*,\bm v_1^*,\cdots,\bm v_k^*)$ is a stationary point of the dynamics \cref{HiOSD2}, we have
	\begin{equation}\label{dotx}
	\bm 0=\dot{\bm x}=\left(I-2\sum_{i=1}^{k}\bm v^*_i {\bm v^*_i}^\top\right)\bm F(\bm x^*).
	\end{equation}
	Then we have $\bm F(\bm x^*)=\bm 0$, indicating that $x^*$ is a stationary point of the dynamical system \cref{dynamical} and $\mathbb{J}^*$ is block lower triangular.
	With $a_{ij}=\left\langle J^*\bm v_j^*, \bm v_i^*\right\rangle$, we obtain
	\begin{equation}\label{Jvt}
	\bm 0=\dot{\bm v}_i=J^*\bm v_i^*-a_{ii}\bm v^*_i-\sum_{j=1}^{i-1} (a_{ij}+a_{ji})\bm v^*_j,\quad i=1,2,\cdots,k.
	\end{equation}
	Therefore, for $j<i$,
	\begin{equation}\label{ajt}
	a_{ji}=\left\langle J^*\bm v^*_i, \bm v^*_j\right\rangle
	=\left\langle \sum_{l=1}^{i-1} (a_{il}+a_{li})\bm v^*_l+a_{ii}\bm v^*_i, \bm v^*_j\right\rangle =a_{ij}+a_{ji}
	\end{equation}
	indicates that the matrix $A_k=(a_{ij})\in\mathbb{R}^{k\times k}$ is upper triangular.
	Since all eigenvalues of $J^*$ are real, we can similarly extend $\{\bm v^*_j\}^k_{j=1}$ to an orthonormal basis $\{\bm v^*_i\}^n_{i=1}$ so that $A=(a_{ij})={V^*}^\top J^*{V^*}$ is an $n$-by-$n$ upper triangular matrix where $V^*=\begin{bmatrix} \bm v_1^*,\cdots,\bm v_n^*\end{bmatrix}$.
	The eigenvalues of $J^*$ are $\{a_{11}, \cdots, a_{nn}\}$, the diagonal of $A$.
	
	With a derivation similar to \cref{VTGxxV}, the eigenvalues of $\mathbb{J}^*_{xx}$ are $-a_{11}, \cdots, -a_{kk}$, $a_{k+1,k+1}, \cdots, a_{nn}$, which are all negative from the linear stability of $(\bm x^*,\bm v_1^*,\cdots,\bm v_k^*)$.
	Therefore, $\bm x^*$ is a $k$-saddle of the dynamics \cref{dynamical}.
	Similarly to \cref{VTGV}, for $i=1,\cdots,k$, from the eigenvalues of $\mathbb{J}^*_{ii}$,
	\begin{equation}\label{eigengii}
	\{-a_{11}-a_{ii}, \cdots, -a_{i-1,i-1}-a_{ii}, -2a_{ii},
	a_{i+1,i+1}-a_{ii}, \cdots,a_{nn}-a_{ii}\},
	\end{equation}
	are all negative, we have $a_{ii}>a_{i+1,i+1}$.
	Therefore, the diagonal elements of $A$ satisfy $a_{11}>\cdots>a_{kk}>0$, and consequently $a_{ii}=\lambda_i$ for $i=1,2,\cdots,k$.
	Furthermore, from \cref{Jvt}, $\{\bm v^*_j\}^i_{j=1}$ is an orthonormal basis of $\mathrm{span}\{\bm w_1, \cdots, \bm w_i\}$ for $i=1,2,\cdots,k$, which completes our proof.
\end{proof}

\emph{Remark:}
In the above theorem, all eigenvalues of $J(\bm x^*)$ are assumed to be real to simplify the derivations.
These assumptions can be relaxed as the eigenvalues $\lambda_{k+1},\cdots,\lambda_n$ have nonzero real parts.
Specifically, assumed that $\lambda_1>\cdots>\lambda_k>\mathrm{Re}\lambda_{k+1}\geqslant\cdots\geqslant\mathrm{Re}\lambda_n$, are nonzero, the proof can be accomplished similarly.

Assumed that the eigenvalues $\lambda_1,\cdots,\lambda_k$ are real, the only difference between GHiSD \cref{HiOSD2} and the HiSD \cref{HiOSD1} results from $J(\bm x)\neq J(\bm x)^\top$.
However, if $J^*$ has a pair of  complex eigenvalues with positive real parts, $(\bm x^*,\bm v_1,\cdots,\bm v_k)$ cannot be a stable stationary point of GHiSD \cref{HiOSD2} for any $(\bm v_1,\cdots,\bm v_k)$.
In spite of the loss of stable convergence of each $\bm v_i$, the subspace $\mathrm{span}\{\bm v_1,\cdots,\bm v_k\}$ can still converge stably to $\mathcal{W}^{\mathrm{u}}(\bm x)$, which also achieves our goal.
Therefore, GHiSD \cref{HiOSD2} can also be applied to search hyperbolic high-index saddle points with complex eigenvalues.

\subsection{Numerical algorithms of GHiSD}
Now we consider the numerical implementation of GHiSD.
The dynamics \cref{HiOSD2v} of GHiSD iterates $\bm v_i$ along $J(\bm x)\bm v_i$ while maintaining the orthonormal condition \cref{orthonormal2}.
Applying an explicit Euler scheme with a sufficiently small step size $\beta>0$, we can establish $\mathcal{W}^{\mathrm{u}}(\bm x)$ via
\begin{equation}\label{v}
\left\{\begin{aligned}
&\tilde{\bm v}^{(m+1)}_i=\bm v^{(m)}_i+\beta J(\bm x) \bm v^{(m)}_i,\qquad i=1,\cdots,k,\\
&\left[\bm v_1^{(m+1)},\cdots,\bm v_k^{(m+1)}\right]= \mathrm{orth}\left(\left[\tilde{\bm v}_1^{(m+1)},\cdots,\tilde{\bm v}_k^{(m+1)}\right]\right).
\end{aligned}\right.
\end{equation}
which can be regarded as a power method of $I+\beta J(\bm x)$ to calculate the eigenvectors of the eigenvalues of $J(\bm x)$ with positive real parts at an exponential rate. 
Here $\mathrm{orth}(\cdot)$ is a normalized orthogonalization function, which can be realized by a Gram--Schmidt procedure.

Since the Jacobian is usually expensive to compute and store in practice, we adopt the dimer method to avoid evaluating the Jacobian explicitly \cite{henkelman1999dimer}.
A dimer centered at $\bm x$ with a direction of $\bm v_i$ and length $2l$ is applied to approximate $J(\bm x)\bm v_i$ in \cref{v}:
\begin{equation}\label{vdimer}
\left\{\begin{aligned}
&\tilde{\bm v}^{(m+1)}_i=\bm v^{(m)}_i+\beta \dfrac{\bm F(\bm x+l\bm v^{(m)}_i) -\bm F(\bm x-l\bm v^{(m)}_i)}{2l} ,\qquad i=1,\cdots,k,\\
&\left[\bm v_1^{(m+1)},\cdots,\bm v_k^{(m+1)}\right]= \mathrm{orth}\left(\left[\tilde{\bm v}_1^{(m+1)},\cdots,\tilde{\bm v}_k^{(m+1)}\right]\right),
\end{aligned}\right.
\end{equation}
which is essentially a central difference for directional derivatives and $l>0$ is a small constant.
Eventually, by applying an explicit Euler scheme to the dynamics \cref{HiOSD2x} of GHiSD, we obtain a numerical algorithm for searching $k$-saddle of non-gradient systems:
\begin{equation}\label{HiOSD3}
\left\{\begin{aligned}
& \bm x^{(m+1)} = \bm x^{(m)}+
\alpha \Big(\bm F(\bm x^{(m)})-2\sum_{j=1}^{k}
\left\langle \bm F(\bm x^{(m)}), \bm v_j^{(m)}\right\rangle \bm v_j^{(m)} \Big), \\
& \tilde{\bm v}_i^{(m+1)} = \bm v_i^{(m)}
+\beta \dfrac{\bm F(\bm x^{(m+1)}+l\bm v^{(m)}_i)-\bm F(\bm x^{(m+1)}-l\bm v^{(m)}_i)}{2l},i=1,\cdots,k, \\
& \left[\bm v_1^{(m+1)},\cdots,\bm v_k^{(m+1)}\right] =
\mathrm{orth}\left(\left[\tilde{\bm v}_1^{(m+1)},\cdots,\tilde{\bm v}_k^{(m+1)}\right]\right),
\end{aligned}\right.
\end{equation}
where $\alpha, \beta>0$ are step sizes. In practice, the discretization scheme \cref{HiOSD3} is implemented with sufficiently small step sizes to ensure numerical stability and convergence. 

\subsection{\label{sec:SL}Construction of solution landscape}
In this subsection, we introduce a systematic numerical procedure to construct the solution landscape of a dynamical system.
This procedure consists of a downward search algorithm and an upward search algorithm, which is a generalization of constructing the pathway map \cite{yin2020construction}.

The downward search is the core procedure to search stationary points starting from a high-index saddle.
Given a $k$-saddle $\hat{\bm x}$, let $\hat{\bm v}_1,\cdots,\hat{\bm v}_k$ denote the orthonormal basis of $\mathcal{W}^{\mathrm{u}}(\hat{\bm x})$ in \Cref{them1}.
We choose a direction $\hat{\bm v}_i$ from the unstable directions $\{\hat{\bm v}_1,\cdots,\hat{\bm v}_k\}$, and slightly perturb the parent state $\hat{\bm x}$ along this direction.
An $m$-GHiSD $(m<k)$ is started from the point $\hat{\bm x}\pm \epsilon \hat{\bm v}_i$ and the $m$ initial directions from the unstable directions need to exclude $\hat{\bm v}_i$.
A typical choice for $m$-GHiSD in a downward search is $(\hat{\bm x} \pm \epsilon \hat{\bm v}_{m+1}, \hat{\bm v}_1,\cdots, \hat{\bm v}_m)$.
This procedure is repeated to newly-found saddles until sinks are found.
The downward search algorithm is presented in detail as \Cref{alg:downward}.

\begin{algorithm}[ht]\caption{Downward search}\label{alg:downward}
	\begin{algorithmic}[1]
		\REQUIRE{$\bm F\in\mathcal{C}^2(\mathbb{R}^n, \mathbb{R}^n)$, a $\hat{k}$-saddle $\hat{\bm x}$ of $\bm F$, $\epsilon>0$.}
		\STATE{Calculate an orthonormal basis $\{\hat{\bm v}_1,\cdots,\hat{\bm v}_{\hat{k}}\}$ of $\mathcal{W}^{\mathrm{u}}(\hat{\bm x})$ using \cref{vdimer};}
		\STATE{Set the queue $\mathcal{A}= \{(\hat{\bm x},\hat{k}-1,\{\hat{\bm v}_1,\cdots,\hat{\bm v}_{\hat{k}}\})\}$,
			the solution set $\mathcal{S}=\{\hat{\bm x}\}$, and the relation set $\mathcal{R}=\varnothing$;}
		\WHILE{$\mathcal{A}$ is not empty}
		\STATE{Pop $(\bm x,m,\{\bm v_1,\cdots,\bm v_k\})$ from $\mathcal{A}$;}
		\STATE{Push $(\bm x,m-1,\{\bm v_1,\cdots,\bm v_k\})$ into $\mathcal{A}$ if $m\geqslant1$;}
		\FOR{$i=1,\cdots,k$}
		\STATE{Determine the initial directions: $\{\bm v_j: j=1, \cdots, m+1, j\neq \min(i,m+1)\}$;}
		\IF{$m$-GHiSD from $\bm x\pm\epsilon \bm v_{i}$ converges to $(\tilde{\bm x}, \tilde{\bm v}_1, \cdots, \tilde{\bm v}_m)$}
		\STATE{$\mathcal{R}\leftarrow\mathcal{R} \cup \{(\bm x,\tilde{\bm x})\}$;}
		\IF{$\tilde{\bm x} \notin \mathcal{S}$}
		\STATE{$\mathcal{S}\leftarrow\mathcal{S} \cup \{\tilde{\bm x}\}$;}
		\STATE{Push $(\tilde{\bm x},m-1,\{\tilde{\bm v}_1,\cdots,\tilde{\bm v}_{m}\})$ into $\mathcal{A}$ if $m \geqslant 1$;}
		\ENDIF
		\ENDIF
		\ENDFOR
		\ENDWHILE
		\ENSURE{The solution set $\mathcal{S}$ and the relation set $\mathcal{R}$.}
	\end{algorithmic}
\end{algorithm}

The downward search algorithm drives a systematic search of stationary points starting from a given parent state in a controlled procedure.
On the other hand, if the parent state is unknown or multiple parent states exist, an upward search algorithm is adopted to search high-index saddles from a sink (or a saddle).
For an upward search, more directions $\hat{\bm v}_1,\cdots,\hat{\bm v}_{K}$ at $\hat{\bm x}$ need to be computed using \cref{vdimer}, where $K>k$ is the highest index of the saddle to search.
Similarly, we choose a direction $\hat{\bm v}_i$ from the stable directions $\{\hat{\bm v}_{k+1},\cdots,\hat{\bm v}_{K}\}$, and start $m$-GHiSD from $\hat{\bm x}\pm \epsilon \hat{\bm v}_i$, where the $m$ initial directions should include $\hat{v}_i$.
A typical choice for $m$-GHiSD $(m>k)$ in an upward search is $(\hat{\bm x} \pm \epsilon \hat{\bm v}_{m}, \hat{\bm v}_1,\cdots, \hat{\bm v}_m)$.
\Cref{alg:upward} presents the upward search algorithm.

\begin{algorithm}[ht]\caption{Upward search}\label{alg:upward}
	\begin{algorithmic}[1]
		\REQUIRE{$\bm F\in\mathcal{C}^2(\mathbb{R}^n, \mathbb{R}^n)$, a $\hat{k}$-saddle $\hat{\bm x}$ of $\bm F$, $\epsilon>0$, the highest searching index $K$.}
		\STATE{Calculate $K$ orthonormal directions $\{\bm v_1,\cdots,\bm v_{K}\}$ at $\bm x$ using \cref{vdimer};}
		\STATE{Set the stack $\mathcal{A}= \{(\hat{\bm x},\hat{k}+1,\{\hat{\bm v}_1,\cdots,\hat{\bm v}_{K}\})\}$ and the solution set $\mathcal{S}=\{\hat{\bm x}\}$,}
		\WHILE{$\mathcal{A}$ is not empty}
		\STATE{Pop $(\bm x,m,\{\bm v_1,\cdots,\bm v_K\})$ from $\mathcal{A}$;}
		\STATE{Push $(\bm x,m+1,\{\bm v_1,\cdots,\bm v_K\})$ into $\mathcal{A}$ if $m<K$;}
		\IF{$m$-GHiSD from $\bm x\pm \epsilon \bm v_m$ converges to $(\tilde{\bm x},\tilde{\bm v}_1,\cdots,\tilde{\bm v}_{m})$}
		\IF{$\tilde{\bm x} \notin \mathcal{S}$}
		\STATE{Calculate an orthonormal basis $\{\tilde{\bm v}_1,\cdots,\tilde{\bm v}_{\tilde{K}}\}$ of $\mathcal{W}^{\mathrm{u}}(\tilde{\bm x})$;}
		\STATE{$\mathcal{S}\leftarrow\mathcal{S} \cup \{\tilde{\bm x}\}$;}
		\STATE{Push $(\tilde{\bm x},m+1,\{\tilde{\bm v}_1,\cdots,\tilde{\bm v}_K\})$ into $\mathcal{A}$ if $m<K$;}
		\ENDIF
		\ENDIF
		\ENDWHILE
		\ENSURE{The solution set $\mathcal{S}$.}
	\end{algorithmic}
\end{algorithm}

By using a combination of downward searches and upward searches, the entire search can navigate up and down systematically to find the complete solution landscape, as long as the saddle points are connected somewhere.

\section{\label{sec:exam}Numerical examples}
\subsection{Three-dimensional example}
We first study a simple three-dimensional (3D) dynamical system to illustrate the solution landscape:
\begin{equation}\label{3d}
\dot{\bm x}=
-\begin{bmatrix}0.6 & 0.1 & 0 \\-0.1 & 0.6 & -0.05 \\0 & -0.1 & 0.6\end{bmatrix}\bm x
+5\begin{bmatrix}(1+(x_1-5)^2)^{-1}\\ (1+(x_2-5)^2)^{-1}\\ (1+(x_3-5)^2)^{-1}\end{bmatrix},
\quad \bm x=\begin{bmatrix}x_1\\ x_2\\ x_3\end{bmatrix}\in\mathbb{R}^3.
\end{equation}
All stationary points of \cref{3d} are listed in \cref{table:tag} and the complete solution landscape is shown in \cref{fig2}.
To search the solution landscape, we first find the source $a_1$ using 3-GHiSD, which is the inverse dynamics of \cref{3d}.
Then we apply the downward search algorithm to construct the complete solution landscape, as shown in \cref{fig2}(B).
\cref{fig2}(A) shows a 3D diagram of the solution landscape, which cannot be visualized in a high-dimensional space.

\begin{table}[h]
	\centering
	\begin{tabular}{cc|cc}
		\hline \hline
		Tag    & Coordinates               & Tag    & Coordinates \\ \hline
		\cellcolor{myca} $a_1$  & $(4.1198,3.4539,3.7131)$  & \cellcolor{mycc} $c_7$  & $(-0.3491,3.7831,5.7853)$  \\ \hline
		\cellcolor{mycb} $b_1$  & $(4.0355,1.6896,3.8422)$ &  \cellcolor{mycc} $c_8$  & $(5.5292,5.8847,3.4660)$ \\ \hline
		\cellcolor{mycb} $b_2$  & $(-0.3626,3.8561,3.6793)$  &  \cellcolor{mycc} $c_9$  & $(5.5930,3.4322,1.0816)$   \\ \hline
		\cellcolor{mycb} $b_3$  & $(5.5995,3.1849,3.7347)$   & \cellcolor{mycc} $c_{10}$  & $(4.2222,5.8207,1.6531)$ \\ \hline
		\cellcolor{mycb} $b_4$  & $(4.2233,5.8467,3.4710)$ &  \cellcolor{mycc} $c_{11}$  & $(4.2247,5.8813,5.8445)$ \\ \hline
		\cellcolor{mycb} $b_5$  & $(4.1265,3.6022,1.1193)$  & \cellcolor{mycd} $d_1$  & $(0.2790,0.4730,0.4653)$  \\ \hline
		\cellcolor{mycb} $b_6$  & $(4.1123,3.2900,5.7716)$   & \cellcolor{mycd} $d_2$  & $(5.6382,1.7000,0.7135)$  \\ \hline
		\cellcolor{mycc} $c_1$  & $(0.2136,0.8094,3.8979)$ &  \cellcolor{mycd} $d_3$  & $(0.1779,0.9943,5.7094)$ \\ \hline
		\cellcolor{mycc} $c_2$  & $(5.6253,2.1940,3.8080)$  & \cellcolor{mycd} $d_4$  & $(-0.6987,5.6858,1.6174)$  \\ \hline
		\cellcolor{mycc} $c_3$  & $(4.0148,1.2843,0.6284)$  & \cellcolor{mycd} $d_5$  & $(-0.7089,5.7422,5.8405)$  \\ \hline
		\cellcolor{mycc} $c_4$  & $(4.0496,1.9728,5.7357)$ & \cellcolor{mycd} $d_6$  & $(5.5299,5.8584,1.6631)$ \\ \hline
		\cellcolor{mycc} $c_5$  & $(-0.7032,5.7106,3.4882)$  & \cellcolor{mycd} $d_7$  & $(5.5283,5.9203,5.8456)$  \\ \hline
		\cellcolor{mycc} $c_6$  & $(-0.3769,3.9328,1.1935)$  & & \\ \hline \hline
	\end{tabular}
	\caption{The stationary points of the 3D example.
		Both the tag and the color of each point specify its index: $a$ for a source, $b$ for a 2-saddle, $c$ for a 1-saddle, and $d$ for a sink.}
	\label{table:tag}
\end{table}
\begin{figure}[h]
	\centering
	\includegraphics[width=1\columnwidth]{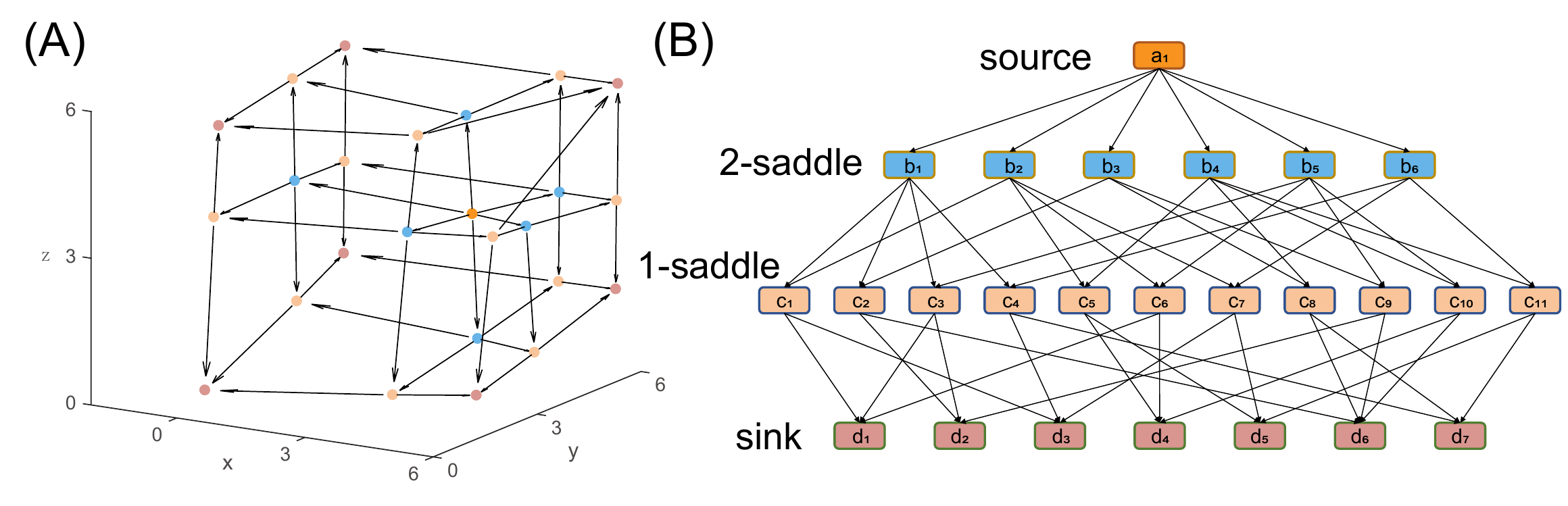}\\
	\caption{(A) A 3D diagram of the solution landscape.
		(B) The solution landscape of the 3D example.
		Each node represents a stationary point, and each arrow represents a dynamical pathway of GHiSD.}
	\label{fig2}
\end{figure}

The solution landscape not only shows the relationships between different sinks, but also reveals rich information on the pathways of the dynamical system.
For example, $d_1$ and $d_2$ are two neighbour sinks connected by the 1-saddle $c_3$.
However, $d_1$ and $d_7$ are not neighbours, and the transition pathway from $d_1$ to $d_7$ needs to pass a metastable state $d_3$ through the pathway sequence $d_1 \rightarrow c_{1} \rightarrow d_3 \rightarrow c_{4} \rightarrow d_7$.
In the solution landscape, it is easy to find that $d_1$ and $d_7$ are connected by a 2-saddle $b_1$, and we can choose a pathway sequence $d_1 \rightarrow c_1 \rightarrow b_1 \rightarrow c_{4} \rightarrow d_7$ to avoid dropping into a sink.

\subsection{Phase field model}
We consider the phase field model as the second numerical example.
The phase field models have been widely employed to investigate nucleation and microstructure evolution in phase transformations \cite{chen2002phase, han2020pathways, zhang2007morphology, zhang2010simultaneous, zhang2016recent}.
Here we consider a phase field model of the order parameter $\phi$ in a fixed square domain $\Omega=[0,1]^2$ with the periodic boundary condition.
The Ginzburg--Landau free energy is
\begin{equation}\label{GLenergy}
E(\phi)=\int_{\Omega}\left(\dfrac{\kappa}{2}|\nabla \phi|^2+\frac{1}{4}(1-\phi^2)^2 \right) \mathrm{d}\bm x,
\end{equation}
where $\kappa>0$ is the gradient-energy coefficient for isotropic interfacial energy, which is a critical parameter that determines the solution landscape.
A finite difference scheme of mesh grids $64\times 64$ is applied to spatially discretize $\phi$ in following numerical simulations. We have tested that the solution landscapes, including solutions and their indices, remain the same under the mesh refinements.

\subsubsection{Gradient system}
We first consider the Allen--Cahn equation of \cref{GLenergy},
\begin{equation}\label{AC}
\dot \phi =-\dfrac{\delta E}{\delta \phi}=\kappa\Delta\phi+\phi-\phi^3,
\end{equation}
which is a gradient system.
For any $\kappa>0$, three homogeneous phases, namely, $\phi_{1}\equiv1$, $\phi_{0}\equiv0$, and $\phi_{-1}\equiv-1$, remain to be stationary points of \cref{AC}.
By analyzing the spectral set of the Hessian $\nabla^2 E(\phi)=3\phi^2-1-\kappa\Delta$, both $\phi_{1}$ and $\phi_{-1}$ are minima of \cref{GLenergy}, while the index of $\phi_{0}$ increases as $\kappa$ decreases to zero, as shown in \cref{table:index}.

\begin{table}[H]
	\centering
	\begin{tabular}{c|ccccc}
		\hline
		$1/\kappa$ & $(0,4\pi^2)$& $(4\pi^2,8\pi^2)$& $(8\pi^2,16\pi^2)$& $(16\pi^2,20\pi^2)$&  $(20\pi^2,+\infty)$  \\  \hline
		index of $\phi_0$ & 1 & 5 & 9 & 13 & $\geqslant$ 21 \\
		\hline
	\end{tabular}
	\caption{Index of $\phi_0$ at different $\kappa$ for the Ginzburg--Landau free energy.}
	\label{table:index}
\end{table}

Starting from the stationary point $\phi_0$, we can obtain multiple stationary points by \cref{alg:downward}.
The solution landscapes at some different $\kappa$ are shown in \cref{fig:kappa}.
Each small image denotes a stationary point of the dynamical system \cref{AC}, and each solid arrow represents a GHiSD (HiSD in this case) pathway.
In each subfigure, an upper state has a higher energy than a lower state.
Furthermore, if $\hat \phi$ is a stationary point of \cref{AC}, so is $-\hat\phi$.
Because of the periodic boundary condition, a translation of a nonhomogeneous stationary point is also a stationary point, so only one phase is shown in the solution landscape for simplicity.
For the simplest case of $\kappa>(4\pi^{2})^{-1}\approx0.0253$, $\phi_{0}$ is a $1$-saddle connecting $\phi_{1}$ and $\phi_{-1}$, as shown in \cref{fig:kappa}(A).

\begin{figure}[h]
	\centering
	\includegraphics[width=0.95\linewidth]{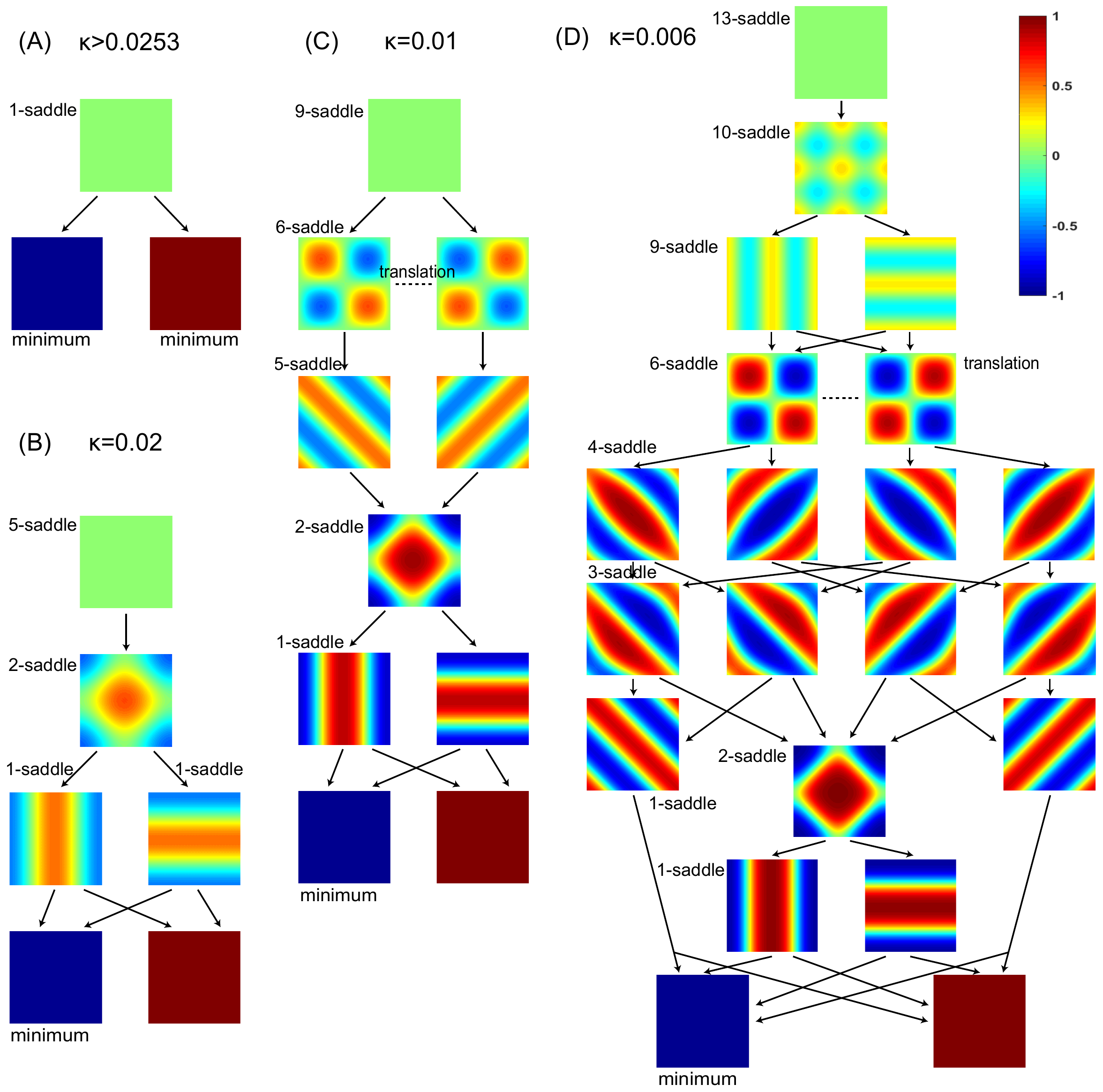}
	\caption{Solution landscapes for the phase field model at different $\kappa$.
		(A) $\kappa>0.0253$.
		(B) $\kappa=0.02$.
		(C) $\kappa=0.01$.
		(D) $\kappa=0.006$. }
	\label{fig:kappa}
\end{figure}

At $\kappa=0.02$, $\phi_0$ becomes a 5-saddle.
The 1-saddles between $\phi_{1}$ and $\phi_{-1}$ are two kinds of lamellar phases, horizontal and vertical respectively (related by a $\pi/2$ rotation), and the two 1-saddles are further connected to a 2-saddle.
The periodic boundary condition implies that the Hessian at a nonhomogeneous state has one or two zero eigenvalues.
Specifically, the Hessian at the 1-saddle has one zero eigenvalues, while the Hessian at the 2-saddle has two.
In \cref{fig:kappa}(B), only one representative stationary point of each state is shown in the solution landscape.
The emergence of zero eigenvalues also explains why no 3-saddles or 4-saddles are found in this system.

As $\kappa$ decreases, more stationary points emerge in this system and the solution landscape becomes more complicated.
When $\kappa$ decreases to $0.01$, $\phi_0$ becomes a 9-saddle, and the solution landscape is shown in \cref{fig:kappa}(C).
In order to visualize solution landscapes better, the dashed line represents stationary points related by a translation.
At $\kappa=0.006$, $\phi_{0}$ becomes a $13$-saddle.
Besides the two kinds of lamellar phases along axes, there are another two 1-saddles between $\phi_1$ and $\phi_{-1}$, which are lamellar phases along $x=y$ and $x=-y$ respectively.
These 1-saddles are bifurcated from the 5-saddles at $\kappa=0.02$ via several pitchfork bifurcations, and  may be ignored with traditional methods because of their relatively high energy.
With the help of the solution landscape, these stationary points can be easily obtained and clearly illustrated.

\subsubsection{Non-gradient system}
By adding a shear flow to the dynamics \cref{AC}, we can obtain a non-gradient dynamical system of the phase field model.
We consider a shear flow of a pushing force along the $x$-axis, as shown in \cref{fig:SF}.
In the presence of the shear flow, the corresponding non-gradient dynamics is
\begin{equation}\label{SF}
\dot \phi =\kappa\Delta\phi+\phi-\phi^3+\gamma \sin(2\pi y) \;\partial_x\phi,
\end{equation}
where $\gamma$ is the shear rate \cite{gu2018simplified, heymann2008pathways}.
The dynamics \cref{SF} also preserves the symmetry of $\phi \to - \phi$.
Because of the shear flow, a translation along the $x$-axis of a stationary point remains a stationary point, so Hessians at nonhomogeneous (along the $x$-axis) stationary points have one zero eigenvalue and only one phase is shown in the solution landscape in most cases as well.
With the parameter $\kappa$ fixed as $0.01$ to compare with \cref{fig:kappa}(C), we increase the shear rate $\gamma$ from zero gradually to obtain different solution landscapes.

\begin{figure}[h]
	\centering
	\includegraphics[width=0.3\linewidth]{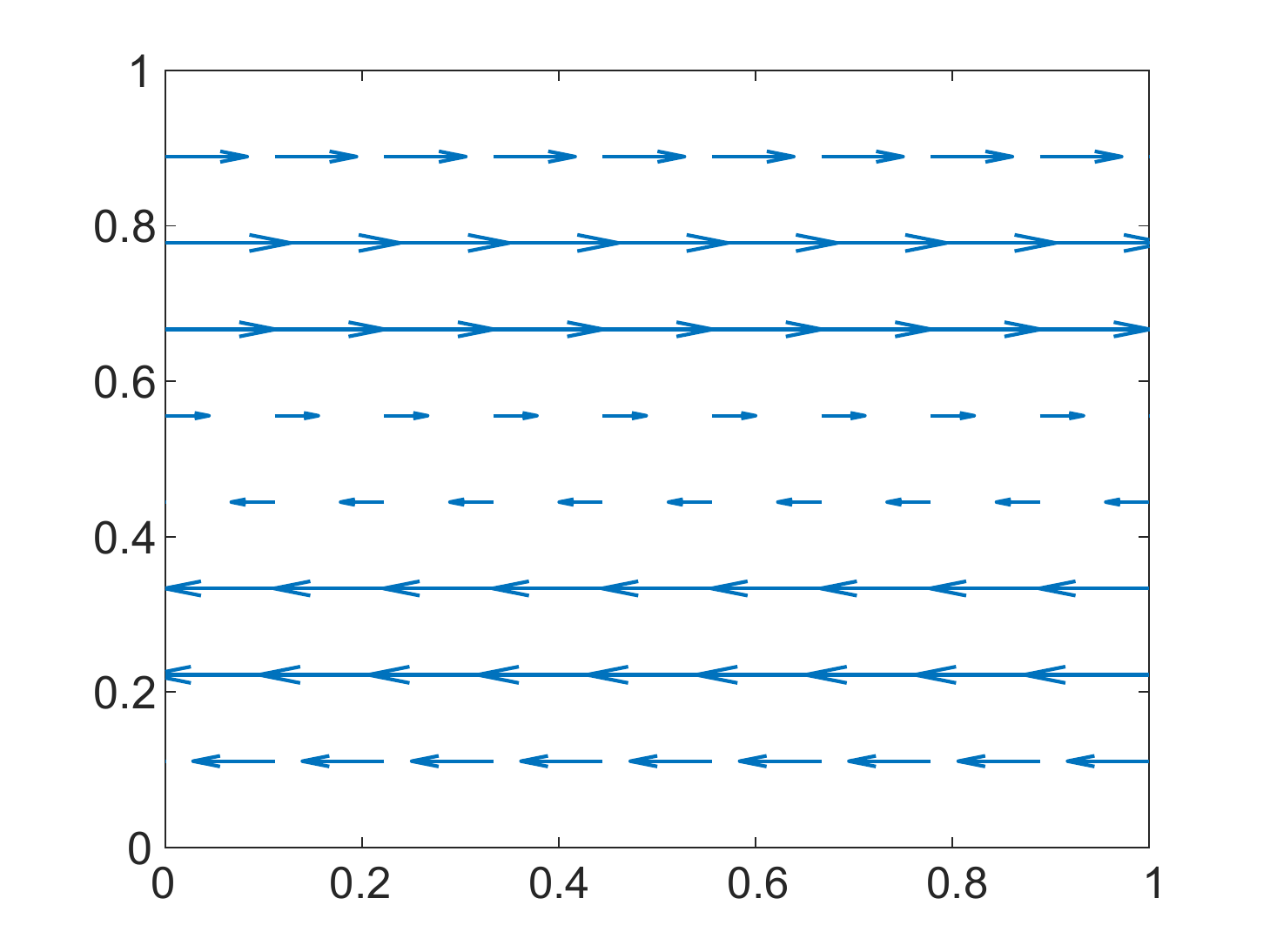}
	\caption{An illustration of the shear flow.}
	\label{fig:SF}
\end{figure}

For any $\gamma>0$, the three homogeneous states and the horizontal lamellar phase remain stationary points of \cref{SF}.
Starting from the stationary point $\phi_0$, we can obtain multiple stationary points by \Cref{alg:downward}, and the solution landscapes at different $\gamma$ are presented in \cref{fig:gamma}.
At $\gamma=0.02$, more stationary points emerge in this system, as shown in \cref{fig:gamma}(A).
The appearance of the 3-saddles and 7-saddles can be explained via pitchfork bifurcations from the 2-saddle and the 6-saddle at $\gamma=0$ in \cref{fig:kappa}(C).
Besides, the two $5$-saddles (related by a $\pi/2$ rotation) and a $1$-saddle (the vertical lamellar phase) are twisted.
When $\gamma$ increases to $0.04$, the 6-saddles at $\gamma=0.02$ vanish first because a pitchfork bifurcation takes place between the $6$-saddle and $5$-saddles as illustrated in \cref{fig:bifurcations}(A).

\begin{figure}
	\centering
	\includegraphics[width=0.95\linewidth]{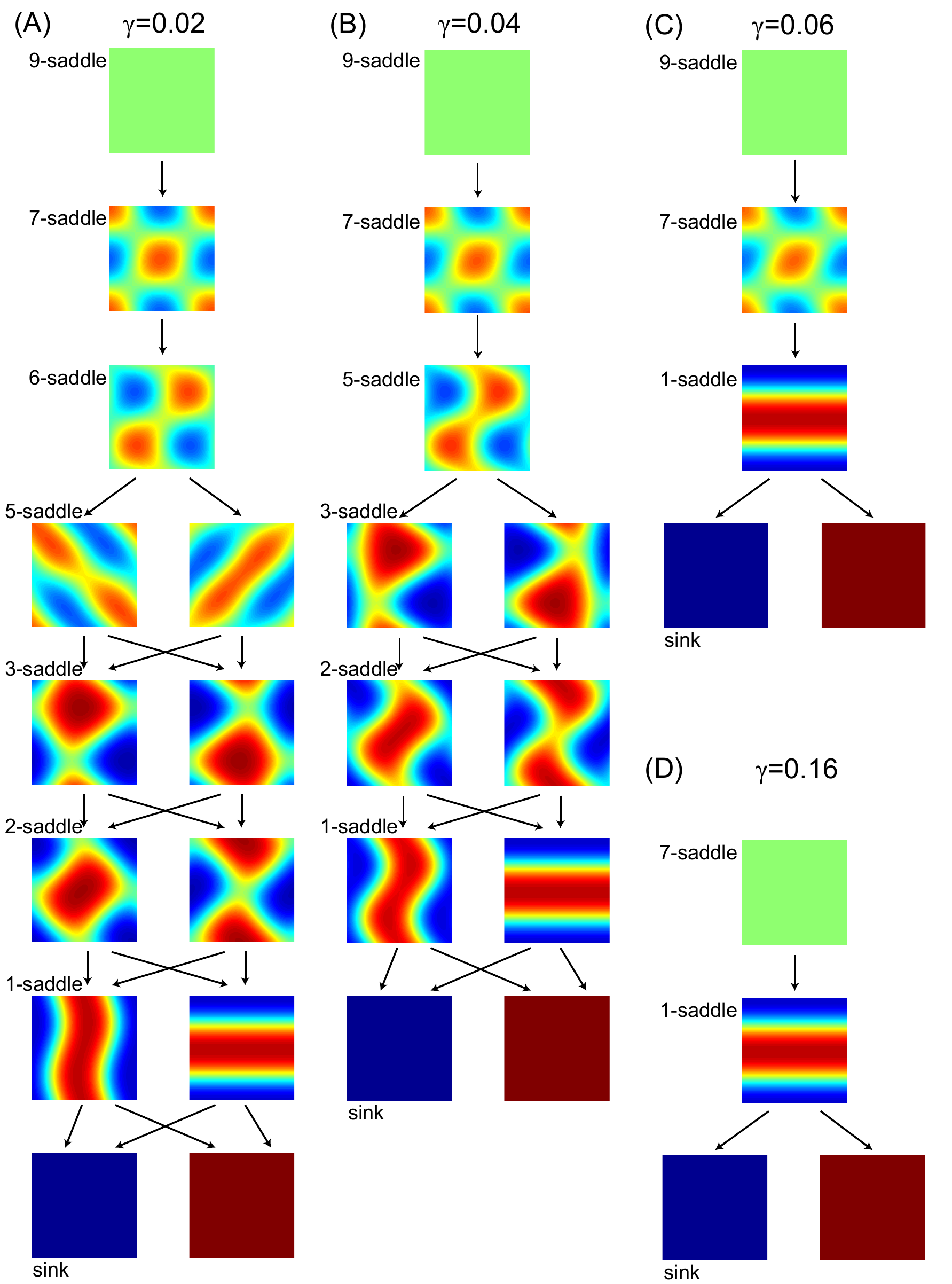}
	\caption{Solution landscapes for the phase field model at $\kappa=0.01$ with different shear rates.
		(A) $\gamma=0.02$.
		(B) $\gamma=0.04$.
		(C) $\gamma=0.06$.
		(D) $\gamma=0.16$.}
	\label{fig:gamma}
\end{figure}

\begin{figure}[h]
	\centering
	\includegraphics[width=1\linewidth]{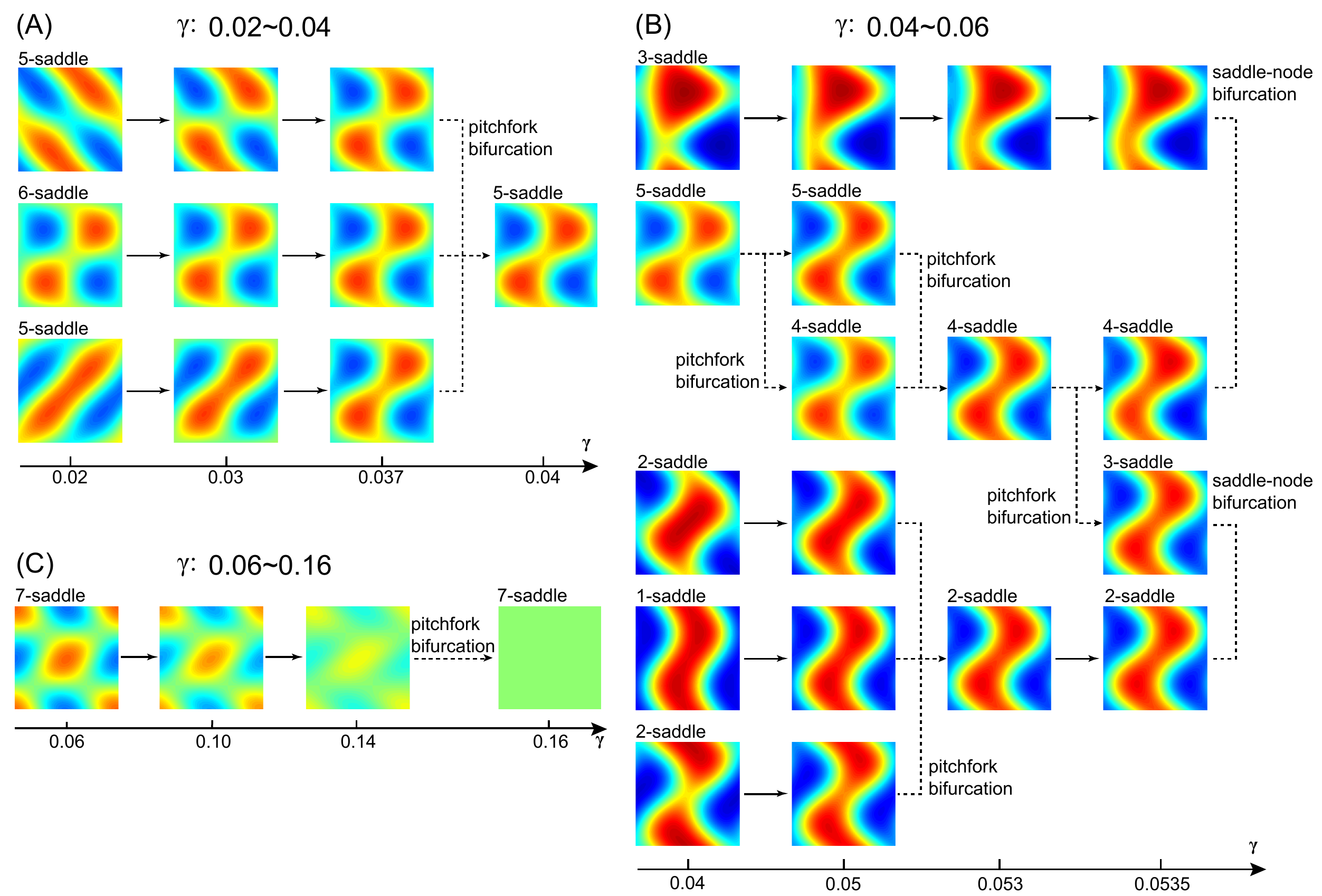}
	\caption{Bifurcations for the phase field model at $\kappa=0.01$ with increasing shear rates.
		(A) $0.02 \le \gamma \le 0.04$.
		(B) $0.04 \le \gamma \le 0.06$.
		(C) $0.06 \le \gamma \le 0.16$.}
	\label{fig:bifurcations}
\end{figure}

When $\gamma$ becomes larger, the solution landscape is simplified from the vanishing of multiple stationary points.
The twisted $1$-saddle no longer exists at $\gamma=0.06$, which has also been discovered by simplified GAD \cite{gu2018simplified}.
This can be well explained with a pitchfork bifurcation between the twisted $1$-saddle and $2$-saddles as illustrated in \cref{fig:bifurcations}(B).
It is worth pointing out that the two lamellar phases go through different and separate changes as $\gamma$ increases.
Furthermore, three other pitchfork bifurcations and two saddle-node bifurcations which occur in $\gamma\in(0.04, 0.06)$ are illustrated in \cref{fig:bifurcations}(B) as well.
As $\gamma$ increases to $0.16$, the shear flow becomes dominant and the solution landscape is shown in \cref{fig:gamma}(D).
The 7-saddles at $\gamma=0.06$ vanish via a pitchfork bifurcation which turns $\phi_0$ into a 7-saddle, as illustrated in \cref{fig:bifurcations}(C).
Thus, all the remaining stationary points at $\gamma=0.16$ are homogeneous along the $x$-axis.

\section{Conclusions and discussions}\label{sec:conclusions}
A long-standing and fundamental problem in computational mathematics and physics is how to find the family tree of all possible stationary points.
In this paper, we introduce the GHiSD method to search the solution landscape for dynamical systems.
The concept of the solution landscape describes the pathway map that starts with a parent state and then relates the entire family completely down to the minima (sinks).
It not only guides our understanding of the relationships between the minima and the transition states on the energy landscape, but also provides a full picture of the entire family of stationary points in dynamical systems.

The GHiSD method is a generalization of the HiOSD method for gradient systems.
It is formulated under the framework of dynamical systems and applicable to search high-index saddle points in both gradient systems and non-gradient systems.
With the GHiSD method, we can efficiently construct the solution landscape using the downward search and upward search algorithms to find multiple possible pathways.
We use a 3D system and the phase field model with a shear flow as numerical examples to illustrate the solution landscapes in dynamical systems.

Advantages of the proposed method are tremendous.
First, it overcomes the difficulty of tuning initial guesses to search the stationary points of dynamical systems required in many other existing methods.
Letting the system gently roll off the high-index saddle along its unstable directions, we have an efficient and controlled procedure to find connected lower-index saddles and sinks with the help of GHiSD.
Second, our method offers a general mechanism for finding all possible sinks (hence the global minima for energy function) without limitations on system types, as long as the system has a finite number of stationary points.
Third, the solution landscape not only identifies the stationary points, but also shows the relationships between all stationary points through the pathway map.
More importantly, this hierarchy structure of stationary points reveals rich and hidden physical properties and processes, leading to a deep understanding of physical systems.

Nevertheless, there are several open questions for the solution landscape.
The parent state is the highest-index saddle and plays a critical role in the solution landscape, but how do we find the parent state in the solution landscape in general?
Apparently the parent state is not unique in a finite-dimensional system with multiple local maxima, but it may be unique in an infinite-dimensional system.
An observation is that the unique parent state often has a better symmetry.
For example, the parent state is the homogeneous state $\phi_0$ in the phase field model and the well order-reconstruction solution in the Landau--de Gennes free energy in a square box \cite{yin2020construction}.
Moreover, can limit cycles, which are very important in the study of dynamical systems, be identified under the framework of GHiSD?
How do we construct the solution landscape in a constrained manifold?
Can the evolution of the solution landscape be explained and predicted with the bifurcation theory?
Both theoretical and numerical investigations of these questions will be of great use.

The concept of the solution landscape can be a critical key for many mathematical problems raised from physics, chemistry and biology.
All applications of the energy landscape may be solved by the solution landscape.
For example, the energy landscape of protein folding has been widely studied in decades \cite{leeson2000protein, mallamace2016energy, onuchic1997theory, wales2003energy}.
One unsolved puzzle is ``folding mechanism''.
Cosmological timescales are required for each protein to find the folded configuration in existing models.
However, naturally occurring proteins fold in milliseconds to seconds.
This paradox of protein folding may be solved by the solution landscape, which is able to capture dynamical pathways on the complicated energy landscape of protein folding.

\bibliographystyle{acm}
\bibliography{GHiSDbib}

\end{document}